\documentclass[12pt,reqno]{amsart}
\usepackage{amssymb}
\usepackage{mathtools}
\usepackage[english]{babel}
\usepackage{epsfig}
\usepackage{mathptmx}
\usepackage{xcolor}
\usepackage{amsmath,amsfonts,amssymb}
\usepackage{mathtools}
\usepackage{mathrsfs}
\usepackage[all]{xy}
\usepackage{graphicx}
\usepackage{latexsym}
\usepackage{verbatim}

\numberwithin{equation}{section}

\def\Re{{\sf Re}\,}
\def\Im{{\sf Im}\,}

\newcommand{\D}{\mathbb D}

\newcommand{\R}{\mathbb R}

\newcommand{\C}{\mathbb C}

\newcommand{\G}{\mathcal{G}}

\newcommand{\B}{\mathbb B}

\newcommand{\oD}{\overline{\mathbb D}}

\def\Re{{\sf Re}\,}
\def\Im{{\sf Im}\,}

\def\id{{\sf id}}

\def\Re{{\sf Re}\,}
\def\Im{{\sf Im}\,}

\def\re{\operatorname{Re}}

\def\Re{{\sf Re}\,}
\def\Im{{\sf Im}\,}

\def\1#1{\overline{#1}}
\def\2#1{\widetilde{#1}}
\def\3#1{\widehat{#1}}
\def\4#1{\mathbb{#1}}
\def\5#1{\frak{#1}}
\def\6#1{{\mathcal{#1}}}
\def\C{\mathbb{C}}
\def\R{\mathbb{R}}
\def\E{\mathbb{E}}

\def\D{\mathbb{D}}
\def\B{\mathbb{B}}

\def\G{\mathbb{G}}

\def\Tt{\mathcal{T}}
\def\Hh{\mathcal{H}}
\def\Oo{\mathcal{O}}
\def\Mm{\mathcal{M}}

\def\id{\operatorname{id}}

\def\re{\operatorname{Re}}

\def\ball{\operatorname{ball}}

\def\Re{{\sf Re}\,}
\def\Im{{\sf Im}\,}

\def\m{\mathcal}
\def\mb{\mathbb}

\newcommand{\mcite}[1]{\csname b@#1\endcsname}

\theoremstyle{theorem}

\def\id{{\sf id}}

\def\Re{{\sf Re}\,}
\def\Im{{\sf Im}\,}

\def\Fix{\operatorname{Fix}}
\def\Deck{\operatorname{Deck}}

\newtheorem{theorem}{Theorem}[section]
\newtheorem{lemma}[theorem]{Lemma}
\newtheorem{proposition}[theorem]{Proposition}
\newtheorem{corollary}[theorem]{Corollary}

\theoremstyle{definition}
\newtheorem{definition}[theorem]{Definition}
\newtheorem{example}[theorem]{Example}

\theoremstyle{remark}
\newtheorem{remark}[theorem]{Remark}

\numberwithin{equation}{section}

\title{Holomorphic retracts in the Lie ball and the tetrablock}

\author[G. Ghosh]{Gargi Ghosh}\email{gargi.ghosh@uj.edu.pl}\author[W. Zwonek]{W\l odzimierz Zwonek}\email{wlodzimierz.zwonek@uj.edu.pl}

\address{Institute of Mathematics, Faculty of Mathematics and Computer Science, Jagiellonian University, \L ojasiewicza 6, 30-348 Krak\'ow, Poland}

\long\def\REM#1{\relax}

\begin{document}
\maketitle

\selectlanguage{english}
\begin{abstract}
 In this article, we study various properties of holomorphic retracts in Lempert domains. We associate the existence and the related form of holomorphic retracts with the linear ones, provide non-trivial examples and discuss their properties in a quite general setting. Later we specialize on two Lempert domains which are the Lie ball of dimension three and its $2$-proper holomorphic image, that is, the tetrablock and give a complete description of holomorphic retracts in these domains. 
\end{abstract}
\section{Introduction}
\footnote{This work is supported by the project No. 2022/45/P/ST1/01028 co-funded by the
National Science Centre and the European Union Framework Programme for Research and Innovation Horizon 2020 under the Marie Sklodowska-Curie grant agreement No. 945339. For the purpose of Open Access, the author has applied a CC-BY public
copyright licence to any Author Accepted Manuscript (AAM) version arising from this
submission.}
Let $D \subseteq \mb C^n$ be a domain. The subset $V \subseteq D$ is said to be a (holomorphic) {\it retract} of $D$ if there exists a holomorphic mapping $R:D\to D$ such that $V=R(D)$ and $R|_V = {\rm id}$ and the mapping $R$ is a {\it holomorphic retraction} associated to the retract $V$. The study of holomorphic retracts of domains in $\mb C^n$ has been done in many settings, for instance \cite{Aba 1989, Agl-Lyk-You 2019, Hea-Suf 1981}.  In this article, we focus on the retracts of Lempert domains (domains on which the methods of the Lempert theory holds) and provide a number of results for this general class of domains. Later, we specialize on $L_3$ to obtain a complete characterization of holomorphic retracts of it, where $L_n$ denotes the classical Cartan domain of type IV (or Lie ball) of dimension $n$.





\subsection{Main results} Now we state our first main result of this article (cf. Proposition~\ref{lemonegeo} and Theorem~\ref{l3retract2dim}).

\begin{theorem}\label{l3re} Let $M$ be a holomorphic retract of $L_3$. If $M$ is a two dimensional retract then $M$ is biholomorphic to $L_2\times\{0\}$. Otherwise, $M$ is a point or a complex geodesic or $L_3$.
\end{theorem}

More precisely, any two dimensional retract $M$ of $L_3$ equals $\Phi(L_2\times\{0\})$, where $\Phi$ is a holomorphic automorphism of $L_3$. Also, the open unit bidisc $\mb D^2\subseteq\mb C^2$ is biholomorphic to $L_2.$ In short, the holomorphic retracts of the Lie ball $L_3$ are geometrically: points, discs, bidiscs or $L_3$. Moreover, a direct application of \cite[Theorem 2]{Mac 2020} shows that Theorem~\ref{l3re} describes all (polynomial) extension sets of $L_3.$ %

A complete characterization of all holomorphic retracts of the open unit ball $\mb B_n$ and the open unit polydisc $\mb D^n$ in $\mb C^n$ can be found in \cite[Chapter 8]{Rud 2008} and \cite[p. 130, Theorem 3]{Hea-Suf 1981}, respectively. The retracts of $\mb D^n$ and $\mb B_n$ are nothing but the lower dimensional polydiscs $\mb D^k,\,\, k\leq n$ and unit balls $\mb B_k,\,\, k\leq n,$ respectively. In Theorem \ref{l3re}, we show that $L_3$ exhibits a very similar attribute as $\mb D^n$ and $\mb B_n$. However, this kind of observation cannot be extended for higher dimensional $L_n.$ Although $L_k\times \{0\}^{n-k},\,\,k\leq n$ are retracts of $L_n$, but they are not exhaustive for $n\geq 4$. For example, there exists a retract of $L_{2n},\,\,n\geq2$ biholomorphic to $\mb B_n,$ cf. Subsection \ref{counterex}.

The problem of finding all the retracts in the proper images of Cartan domains is rather recent. The symmetrized bidisc $\mb G_2$ is the image of $\mb D^2$ under the $2$-proper holomorphic map $(z,w) \mapsto (z+w,zw).$ A comprehensive study on the holomorphic retracts and extension sets of the symmetrized bidisc is done in \cite{Agl-Lyk-You 2019}. If $D$ is a bounded domain in $\mb C^3$ and $f:L_3\to D$ is a proper holomorphic map of multiplicity $2$ then $D$ is biholomorphic to the tetrablock $\mb E$ \cite[Proposition 3.4]{Gho-Zwo 2023}.  Motivated by \cite{Agl-Lyk-You 2019}, we extend our observation to find all the retracts in the tetrablock. 

To state the result, we recall that $L_3$ and $R_{III}(2)$ are holomorphically equivalent, where $R_{III}(2)$ denotes the set of all $2\times2$ symmetric complex matrices whose largest singular value is less than $1.$ The $2$-proper holomorphic mapping $\Lambda: R_{III}(2) \to \mb E$ is defined by $$\Lambda(A)=(a_{11},a_{22},a_{11}a_{22}-a_{12}^2)$$ and $\mathcal J(\Lambda)$ denotes the set of critical arguments of $\Lambda.$ 


\begin{theorem} Let $M$ be a holomorphic retract of the tetrablock $\mathbb E$. If $M$ has dimension two, then $M$ is of one of the following forms:
\begin{enumerate}
    \item $\Lambda(N)$, where $N$ is a two-dimensional holomorphic retract in $R_{III}(2)$ such that $N\cap\mathcal J(\Lambda)=\emptyset$ or
    \item $\Phi(N),$ where $\Phi$ is an automorphism of $\mathbb E$ and $N$ is either $\{(a,b,ab):a,b\in\mathbb D\},$ or $\left\{\left(\frac{s}{2},\frac{s}{2},p\right):(s,p)\in\mathbb G_2\right\}$.
\end{enumerate} Otherwise, $M$ is either a point or a complex geodesic or is equal to $\mathbb E$.
\end{theorem}

Consequently, holomorphic retracts of the tetrablock can be the points, discs, bidiscs, symmetrized bidiscs or the tetrablock itself. A very recent article in arXiv \cite{Bal-Mam 2024} provides an updated survey of the theory of retracts. We complement and extend methods from it to provide various general properties of retracts.

    

\subsection{Lempert domains, retracts -- basic information and preliminary results}
Let $D$ be a domain in $\mb C^n.$ For $z,w\in D,$ we define the {\it Lempert function} by $$l_D(z,w) := {\rm inf}\{p(\lambda_1,\lambda_2): \text{ there exists } f\in \m O(\mb D,D) \text{ such that } f(\lambda_1)=z,f(\lambda_2)=w\}$$ and the {\it Carath\'eodory pseudodistance} by $$c_D(z,w):={\rm sup}\{p(F(z),F(w)): F \in \m O(D,\mb D)\},$$ where $\m O(D_1,D_2)$ denotes the set of holomorphic functions from the domain $D_1$ to the domain $D_2$ and $p$ denotes the Poincar\'e distance on the unit disc $\mathbb D.$ In 1981, Lempert proved that $l_D=c_D,$ if $D$ is a convex domain in $\mb C^n$ \cite{Lem 1981}. For basic information on holomorphically invariant functions, their properties and Lempert theorem we suggest the reader to consult \cite{Jar-Pfl 2013}. We call a connected complex manifold $M$ {\it Lempert manifold} if it is taut and $l_M=c_M.$ It is an immediate consequence of the definition that $(M,l_M)$ is a complete metric space. For any two distinct points $z$ and $w$ in a Lempert manifold, there exists a complex geodesic passing through those points. The mapping $f\in\mathcal O(\mathbb D,D)$ is called a {\it complex geodesic} if there exists $F\in\mathcal O(D,\mathbb D)$ such that $F\circ f$ is an automorphism of $\mathbb D$ and such an $F$ is called a {\it left inverse for $f$}. 
With appropriate choice of holomorphic automorphisms of $\mathbb D$, say $a$ and $b,$ we can consider the complex geodesic $f\circ b$ and its left inverse $a\circ F$ such that $a\circ F\circ f\circ b$ is identity. Thus without loss of generality, we assume that $F\circ f$ is the identity and identify the left inverses (or complex geodesics) which are equal up to suitable automorphisms of $\mathbb D$. The problem of (non)-uniqueness of left inverses has been studied in detail in \cite{Kos-Zwo 2016}. For example, bounded convex domains, strongly linearly convex domains, the symmetrized bidisc and the tetrablock are Lempert domains \cite{Abo-You-Whi 2007,Agl-You 2001,Agl-You 2006,Cos 2004,Edi-Kos-Zwo 2013,Lem 1981,Lem 1982,Lem 1984}.  


For a holomorphic self-mapping $F:D\to D,$ we denote the set of fixed points by $$\Fix (F)=\{z\in D:F(z)=z\}.$$ If $D$ is taut, then $\Fix(F)$ is a complex manifold \cite[Section 2.5]{Aba 1989}. In particular, if $R:D\to D$ is a holomorphic retraction and $V=R(D)$ is the associated retract, then clearly $V=\Fix(R)$ is a complex submanifold (when $D$ is taut).

However, for any holomorphic self-mapping $F,$ the set $\Fix (F)$ might not be a retract in general. We provide an elementary example to show it, see also \cite{Fri-Ma-Vig 2006}. If $\Fix(F)$ is a retract, then it is connected. Consider the holomorphic mapping $F$ from the annulus $A(0;1/r,r) \to A(0;1/r,r)$ such that $F(\lambda)= -\lambda$ and note that $\Fix (F)$ is not connected, hence not a retract. 
However, the answer is yes for convex domains \cite{Aba 1989}. 

\begin{remark}\label{rem1} Let $D$ be a Lempert domain. Any holomorphic retract $V\subset D$ is {\it weakly totally geodesic}, that is, for any two points $w,z\in V$ there exists a complex geodesic passing through $w$ and $z$ which is entirely contained in $V$. To see this, we suppose that $f:\mathbb D\to D$ is a complex geodesic passing through $w$ and $z$ and $R:D\to V$ is a retraction. Then the map $R\circ f$ is a complex geodesic in $D$ joining $w$ and $z$ that lies entirely in $V$.\end{remark}



Suppose that $M\subset D$ is a holomorphic retract in a Lempert domain $D$.  It is an interesting problem to determine (up to biholomorphisms) the $k$-dimensional manifolds $M$ for which there are holomorphic mappings $i:M\to D$ and $r:D\to M$ such that $i\circ r$ is the identity. This is a generalization of the notions of complex geodesics and left inverses. Below we see that the one dimensional retracts are determined uniquely and they are complex geodesics (geometrically discs).

\begin{proposition}\label{lemonegeo}
Let $D$ be a Lempert domain and $V$ be a one-dimensional retract in $D$. Then $V$ is the (image) of a complex geodesic.
\end{proposition}
\begin{proof} Let $z$ and $w$ be two distinct points in $V$. There exists a complex geodesic $f:\mathbb D \to D$ that joins $z$ and $w$ and lies entirely in $V,$ cf. Remark \ref{rem1}. Since $V$ is a one-dimensional complex manifold which contains another one-dimensional complex manifold, by the connectivity of both $V$ and $f(\mathbb D),$ we conclude that $V=f(\mathbb D)$.
\end{proof}
We indicate that this result is proved by an elementary argument in a quite general setting. A similar result as in Proposition \ref{lemonegeo} is proved for the symmetrized bidisc in \cite[Theorem 5.1]{Agl-Lyk-You 2019}. On the other hand, under other assumption of simple connectivity of $D,$ the result is proven in \cite[Theorem 2.25]{Bal-Mam 2024}. 

Certainly, if $M\subset D$ is a retract then $M$ is the extension set; in particular $c_M=(c_D)_{|M\times M}$. Recall that $M\subset D$ is called an {\it extension set} if any bounded holomorphic function $f$ on $M$ can be extended to a holomorphic function defined on $D$ with the sup norm preserved.   We also know that a retract $M$  of the Lempert domain $D$ is weakly totally geodesic; thus, $l_M=(l_D)_{|M\times M}$. Therefore, we formulate the following result.

\begin{proposition}\label{proposition:retract-is-lempert} Suppose that $D$ is a Lempert domain. Let  $M\subset D$ be a holomorphic retract, then $M$ is a Lempert manifold.
\end{proposition}
However, the case of the tridisc shows that the converse does not hold in general \cite{Kos-Zwo 2021}. Now we make another relevant observation to relate the problem of uniqueness of left inverses for complex geodesics with the necessary form of retractions.

\begin{proposition}\label{proposition:necessary-retraction}
    Let $D$ be a Lempert domain in $\mathbb C^n$ and $M\subset D$ be a retract with the retraction $R:D \to M$. If there exists a complex geodesic $f$ with its image lying entirely in $M$ which has uniquely determined left inverse $F$ (in $D$), then $$F\circ R(z)=F(z)\,\, \text{ for all } z\in D.$$
\end{proposition}
\begin{proof}
It is sufficient to observe that $F\circ R$ is a left inverse to $f$ and then the result follows from the  assumption of the uniqueness of left inverse.
\end{proof}


\subsection{Linear retracts in the indicatrix provide a necessary form of a retract}
Let $\delta$ be a holomorphically invariant function on $D$ (such as Carath\'eodory-Reiffen pseudometric $\gamma$, Kobayashi-Royden pseudometric $\kappa$).  Basic information of these metrics can be found in \cite{Jar-Pfl 2013}. For $z\in D,$ the set $$I^{\delta}_D(z):=\{X\in\mathbb C^n:\delta_D(z;X)<1\}$$ is called {\it the indicatrix} of $D$ at $z$. 

Let $F:D\to D$ be a holomorphic self-mapping. If $M=\Fix(F)$ is a complex manifold, we have the following inclusion for every $z\in D$
\begin{equation*}
F^{\prime}(z)(I_D^{\delta}(z))\subset I_D^{\delta}(z) \text{ and } F^{\prime}(z)|_{T_zM}=\id_{T_zM}.
\end{equation*}
That is, we get some necessary condition for the structure of $\Fix(F)$. If $F$ is a retraction, we additionally know that the mapping $F^{\prime}(z)$ maps $I_D^{\delta}(z)$ into $T_zM$.  Consequently, $I_D^{\delta}(z)\cap T_z M$ is a {\it linear retract}  of $I_D^{\delta}(z)$. The concept of a linear retract was considered in \cite{Bal-Mam 2024}.
Certainly, if $D$ is a Lempert domain then we may neglect the superscript in $I_D^{\delta}$.


\begin{proposition} Let $D$ be a Lempert domain in $\mb C^n.$ If $F:D\to M\subset D$ is a holomorphic retraction, then $T_zM\cap I_D^{\delta}(z)$ is a linear retract of $I_D^{\delta}(z)$ for every $z\in M.$
\end{proposition}
\begin{remark}\label{remark:reduction-retracts} Let $D$ be a bounded balanced convex domain. Suppose that $M$ is a holomorphic retract passing through $0$. We know that $I_D(0)\cap T_0M$ is a linear retract (and thus holomorphic) of $T_0D=D$. Let $R:T_0D=D\to T_0M\cap D$ be a linear retraction. Then $d_0R=R$ is the identity on $T_0M.$ We use \cite[Theorem 0.1]{Bel 1994} for $R|_M$ to conclude that $R$ is a one-to-one mapping from $M$ onto $D\cap T_0M$. For a similar result compare \cite[Theorem 3.7]{Bal-Mam 2024}.
\end{remark}
If we have the uniqueness of complex geodesics (for example, in strongly linearly convex domains) the necessary form of a holomorphic retract $M$ must be the following $M=\{f(\lambda):f \text{ is a complex geodesic such that $f(0)=z$ and $f^{\prime}(0)\in T_zM$}\}$. This leads us to formulating  the following problem. 

\begin{remark}\label{remark:form-retracts}
Let $D$ be a Lempert domain with uniquely determined geodesics. Fix $w\in D$. Assume that for some $k$-dimensional linear subspace $V\subset\mathbb C^n$ the necessary condition for the holomorphic retract is satisfied, which means that a (convex) set $I_D(w)\cap V$ is a linear retract of $I_D(w)$. Then $M$ must be of the following form
\begin{equation*} M=\{f(\lambda):\lambda\in\mathbb D,\; f:\mathbb D\to D, f(0)=w, f^{\prime}(0)\in V, f\text{ is a complex geodesic in $D$}\}.
\end{equation*}
However, we presume that this necessary condition for the form of the holomorphic retract is not sufficient in general. It is trivially satisfied for the unit ball which allows us to find an immediate way to describe all the retracts in the unit ball. This approach is much simpler compared to the method in \cite{Suf 1975} where the description of the holomorphic retracts of the unit ball is provided. Hence, an interesting question is to determine when this necessary form is sufficient as well.  

Also, we note that even in a more general setting (without assuming the uniqueness of complex geodesics) we can find a candidate for a $k$-dimensional retract passing through the given point $z$ and with the given tangent space $V$. It is a union of (not necessarily all) complex geodesics passing through $z$ in the direction of vectors of $V$. 

The first step towards understanding the problem could be to employ the effective formulas for complex geodesics in convex ellipsoids \cite{Jar-Pfl-Zei 1993}, where the uniqueness of complex geodesics is present. We provide some non-trivial examples of retracts in the ellipsoids in a subsequent section.

\end{remark}

\subsection{Equivalence of retracts with the same tangent space}
 Let $M\subset D$ be a holomorphic retract. There can be more than one retractions $R:D \to M$ associated to $M.$ If $D$ is a Lempert domain, the one-dimensional holomorphic retracts are precisely the (images) of complex geodesics. 
 For the complex geodesic $f:\mathbb D\to D,$ there is a one-to-one correspondence between the retractions $R$ of $f(\mb D)$ and left inverses $F$ of $f$ given by the formula $R(z)=f\circ F(z),\,\,z\in D$. Therefore, in this one-dimensional case the (non)-uniqueness of retractions is equivalent to the (non)-uniqueness of left inverses. The latter problem has been studied, for instance, in \cite{Kos-Zwo 2016}. However, the following result addresses the uniqueness of retracts (up to a biholomorphism) under certain conditions. 


\begin{theorem}\label{theorem:uniqueness-retracts}
Let $D$ be a Lempert domain in $\mathbb C^n$ and $M_1$, $M_2$ be two $k$-dimensional retracts with $w\in M_1\cap M_2.$ Suppose that $R_j:D\to M_j$, $j=1,2$ are the corresponding holomorphic retractions. If $R'_j(w)$ restricted to $T_wM_{3-j}$ is a linear isomorphism onto $T_w M_{j}$ for $j=1,2$ and they are mutually inverse, then $M_1$ and $M_2$ are biholomorphic. Additionally, the two mutually invertible biholomorphisms are given by the formulae $(R_j)_{|M_{3-j}}:M_{3-j}\to M_j$ for $j=1,2.$
\end{theorem}
\begin{proof}
Let us assume that $w=0.$ We define $r_j=R_j\circ (R_{3-j})_{|M_j}:M_j\to M_j$. Then $r_j(0)=0$, $r_j'(0)$ is the identity on $T_0M_j$. We restrict ourselves to the ball $B_{M_j}(0,\epsilon)$ of radius $\epsilon$ with respect to the Kobayashi distance for suitably chosen $\epsilon>0$. For a small enough $\epsilon,$ we get a biholomorphism $\Phi_j:U_j\to B_{M_j}(0,\epsilon)$, where $U_j$ is a bounded domain in $\mathbb C^k$. Then we apply Cartan Theorem for $\Psi_j:=\Phi_j^{-1}\circ r_j\circ\Phi_j:U_j\to U_j$ to conclude that $\Psi_j$ is the identity and 
consequently $R_j$ and $R_{3-j}$ are mutually inverse in a small neighborhood of $0$. The identity principle (with suitable mappings restricted to $M_j$) finishes the proof. 
  \end{proof}

\begin{remark}\label{remark:uniqueness-retracts}
    Suppose that $T_wM_1=T_wM_2=\m X$ (say). Then each $R'_j(w),\,j=1,2,$ restricted to the tangent space $\m X$ is the identity. Hence, $M_1$ and $M_2$ are biholomorphic with the biholomorpisms described as in Theorem~\ref{theorem:uniqueness-retracts}.
\end{remark}
From this point we restrict our consideration to the Lie ball and the tetrablock.

\section{Retracts in the Lie ball}
In this section, we give a description of two-dimensional holomorphic retracts in the three dimensional Lie ball. We start by recalling some definitions and known results.

For $z\in \mb C^n,$ the Euclidean norm of $z$ is denoted by $||z||$ and $z\bullet z:=z_1^2+\ldots+z_n^2.$ Let $\mb B_n=\{z\in \mb C^n: ||z||<1\}.$  For $n\geq 1,$ 
\begin{equation*}
L_n =\{ z\in\mathbb B_n: 2||z||^2-|z\bullet z|^2<1\},
\end{equation*}
 denotes the Lie ball (or the classical Cartan domain of type $IV$) \cite{Ara 1995}. Since $L_n$ is a bounded symmetric domain, the understanding of its holomorphic retracts can be reduced to the ones passing through the origin. Additionally, $L_n$ is a convex and balanced domain in $\mathbb C^n$. The Shilov boundary of $L_n$ is given by $$\partial_S L_n:=\{\omega x: \omega \in \mb C,\,|\omega|=1 \text{ and } x\in\mathbb R^n,x_1^2+\ldots+x_n^2=1\}.$$ In other words, $\partial_S L_n=\mb T\cdot \mathbb S_{\mathbb R}^{n-1},$ where $\mb T$ is the unit circle in $\mb C$ and $S_{\mathbb R}^{n-1}$ is the unit sphere in $\mb R^n.$ Clearly, $\partial_S L_2\times \{0\}\subseteq \partial_S L_3.$
 
 As $L_1=\mathbb D$ and $L_2$ is biholomorphic to the bidisc $\mathbb D^2,$ the first unknown problem for determining of the retracts is $L_3$. Below we give a complete description of retracts in $L_3$. As a good reference for properties of Lie balls, we recommend \cite{Ara 1995}.

\begin{remark}\label{Remfzero}
    Suppose that $M=\{(z_1,z_2,f(z_1,z_2)):(z_1,z_2)\in L_2\}\subset L_3$ for a holomorphic mapping $f:L_2\to\mathbb C.$ Evidently $M$ is a two-dimensional holomorphic retract of $L_3$. Since $L_3\subseteq \mathbb B_3$ we see that $f$ ends to $0$ as the points tend to $\partial_S L_2.$ Consequently, we conclude that $f$ is the zero function.
\end{remark}




    We recall some known facts which are essential for the next proof. The Lie ball $L_2$ is biholomorphic to $\mathbb D^2$ by $\phi:L_2 \to \mb D^2$ defined as follows:
\begin{equation}\label{bidisclie}
\phi(z_1,z_2)= (z_1+iz_2,-z_1+iz_2).
\end{equation}

\begin{remark}
Let $|a|\leq 1$ and $R_a:\mathbb D^2 \to \mathbb D^2$ be defined by $ R_a(z_1,z_2)=(z_1,az_1).$ Any (one-dimensional) linear retract of $\mathbb D^2$ passing through $0$ is of the form $R_a(\mathbb D^2)$ (up to a permutation of variables). Moreover, a complex geodesic $f=(f_1,f_2):\mb D\to \mb D^2$ has a non-uniquely determined left inverse if and only if both $f_1$ and $f_2$ are automorphisms of $\mb D.$ For $|a|<1$, the non-uniquely determined complex geodesics connecting $(0,0)$ and $(\lambda,a\lambda)$ determine the unique left inverse $F_a: \mb D^2 \to \mb D$ satisfying the formula $F_a(z_1,z_2) =z_1.$  Thus the (linear) retraction for the retract $R_a(\mathbb D^2)$ in the bidisc is uniquely determined. On the other hand, for $|a|=1$ the formula for the suitable retraction follows from the description of left inverses for the complex geodesic $\mathbb D\owns\lambda\to(\lambda,a\lambda)\in\mathbb D^2$. All the linear retractions are $R_{(a,t)}:\mathbb D^2 \to \mathbb D^2$ for $t\in [0,1]$, where
\begin{eqnarray}\label{eqna1}
    R_{(a,t)}(z_1,z_2)=(tz_1+(1-t)\overline{a}z_2)(1,a),
\end{eqnarray}
(for example, use \cite[Example 11.79]{Agl-McC 2002}).
\end{remark}

Our next result completes the description of all holomorphic retracts in $L_3$.

\begin{theorem}\label{l3retract2dim} Let $M$ be a two dimensional holomorphic retract of $L_3$. Then $M$ is biholomorphic to $L_2\times\{0\}$. 
\end{theorem}
More precisely, any such a retract $M$ equals $\Phi(L_2\times\{0\})$, where $\Phi$ is a holomorphic automorphism of $L_3$.
\begin{proof} Without loss of generality, we assume that $0\in M$. The proof is twofold. First we show that the problem can be reduced to linear retracts of $L_3$. Then we prove that any two-dimensional linear retract of $L_3$ is linearly equivalent (up to a linear automorphism of $L_3$) to $L_2\times\{0\}.$

  Let $M$ be a two dimensional retract of $L_3$. Then $T_0M\cap L_3$ is a linear retract of $L_3$, cf. Remark~\ref{remark:reduction-retracts}. Moreover, if the retract $T_0M\cap L_3$ is linearly equivalent to $L_2\times\{0\}$ then $M$ must be holomorphically equivalent to the retract of the form $\{(z_1,z_2,f(z_1,z_2)):(z_1,z_2)\in L_2\}\subset L_3$ for a holomorphic mapping $f:L_2\to\mathbb C$ (use Theorem~\ref{theorem:uniqueness-retracts} and Remark~\ref{remark:uniqueness-retracts}). We conclude that the function $f$ must be identically $0$ from Remark \ref{Remfzero}.

It remains to prove that any two-dimensional linear retract of $L_3$ is linearly equivalent (up to a linear isomorphism of $L_3$) to $L_2\times\{0\}$.  Consider a two-dimensional linear retract $M=L_3\cap V$ of $L_3$ with the linear retraction $R:L_3 \to M$ and $V$ being a linear two-dimensional subspace of $\mathbb C^3$.

Up to a linear isomorphism of $L_3,$ we may assume that $R(v_1)=0$ for some $v_1\in\partial L_2\times\{0\}$. Then $V\cap(L_2\times\{0\})$ is one dimensional and spanned by $v_2\in\partial L_2\times\{0\}$. Therefore, $R|_{L_2\times \{0\}}$ is a linear retraction with the retract $\mathbb Cv_2\cap (L_2\times \{0\})$. From the explicit form of linear retracts in $\mathbb D^2$ and Equation \eqref{bidisclie}, we can consider $v_2=\frac{1}{2}(1-a,-i(1+a),0)$ for some $|a|\leq 1.$

Suppose that $|a|<1$. Hence, from the uniqueness of holomorphic retractions in $\mathbb D^2,$ we may assume that $v_1=\frac{1}{2}(1,-i,0)$. Let $v_3=(b_1,b_2,1)$ be a vector such that $v_2$ and $v_3$ span $V$. The point $(0,0,1)=\lambda_1 v_1+\lambda_2 v_2+\lambda_3 v_3 \in \overline{L}_3.$ The linear retraction $R$ must map $(0,0,1)$ to a point with the Euclidean norm $\leq 1.$ Clearly, $\lambda_3=1$ and $(0,0,1)$ can be taken as $v_3$. Let $x=(x_1,x_2,x_3)\in \mb S_{\mb R}^{2}:=\{x\in\mathbb R^3:||x||=1\}.$ Then $Rx\in \overline{L}_3$. We write $$x=\lambda_1\left(\frac{1}{2},\frac{i}{2},0\right)+\lambda_2\left(\frac{1-a}{2},-i\frac{1+a}{2},0\right)+\lambda_3(0,0,1)$$ then $\lambda_1=x_1(1+a)-ix_2(1-a)$, $\lambda_2=x_1+ix_2$ and $\lambda_3=x_3$. We use the description of $L_3$ to conclude that \begin{equation*} ((x_1^2+x_2^2)|1-a|^2+(x_1^2+x_2^2)|1+a|^2+4x_3^2)/2\leq 1+\left|1/4(x_1+ix_2)^2((1-a)^2-(1+a)^2)+x_3^2\right|^2. \end{equation*} \vspace{0.1 in} Consequently, we get \begin{eqnarray*} (x_1^2+x_2^2)(1+|a|^2)+2x_3^2 &\leq& 1+|a|^2(x_1^2+x_2^2)^2+x_3^4-2\re(a(x_1+ix_2)^2x_3^2) \\ &\leq& 1+(|a|(x_1^2+x_2^2)+x_3^2)^2. \end{eqnarray*} Suppose that for some $r\in (0,1),$ $x_1^2+x_2^2=r^2$ and so $x_3^2=1-r^2$. From the above expression, we have $(1-r^2)(1-|a|)^2\leq 0$ which gives a contradiction.  

Let $|a|=1.$ We lose no generality assuming that $a=1$. Hence, the corresponding holomorphic retract of $L_2\times\{0\}$ is the set $\{(0,-i\lambda,0):\lambda\in\mathbb D\}$ (which corresponds to $\mathbb D\owns\lambda\to(\lambda,\lambda)\in\mathbb D^2$). 
The corresponding linear retractions in $L_2\times\{0\}$ are of the form 
\begin{equation*}
(z_1,z_2,0)\to (0,-it(z_1+iz_2)-i(1-t)(-z_1+iz_2),0).
\end{equation*}
Using analogous argument as above, we choose the vectors $v_1=(-i,2t-1,0)$ and $v_2=(0,-i,0)$ and $v_3=(0,0,1)$. Then we easily get that the holomorphic retract is $\{0\}\times L_2$ which finishes the proof.
\end{proof}

\begin{remark}
Recently in \cite{Mac 2020}, Maciaszek proved that the subsets of $L_3$ with polynomial extension  property are the retracts of $L_3$. Thus Theorem \ref{l3retract2dim} and Proposition \ref{lemonegeo} together with \cite[Theorem 2]{Mac 2020} provide a nice description for (polynomial) extension sets in $L_3.$
\end{remark}

\subsection{Retracts in higher dimensional Lie ball -- a counterexample}\label{counterex}
The method used in the proof of Theorem \ref{l3retract2dim} cannot be used for $L_n,\,\,n\geq 4$ to describe the holomorphic retracts (even of dimension two). We consider $L_{2n},\,\,n\geq 2$ and the holomorphic mapping $R$ defined on $\mathbb C^{2n}$ by 
\begin{equation*}
R(z_1,\ldots,z_{2n}):=\left(\frac{z_1-iz_2}{2},\frac{z_2+iz_1}{2},\frac{z_3-iz_4}{2},\frac{z_4+iz_3}{2},\ldots,\frac{z_{2n-1}-iz_{2n}}{2},\frac{z_{2n}+iz_{2n-1}}{2}\right).
\end{equation*}
Note that $R(L_{2n})\subset L_{2n}$. Actually, let $z=(z_1,\ldots,z_{2n})\in L_{2n}$ then so is $(-iz_2,iz_1,\ldots,-iz_{2n},iz_{2n-1})$. Then we use the convexity of $L_{2n}$ to conclude that $R(z)\in L_{2n}$. Note that $R$ is a retraction for the retract $M=R(L_{2n})=\{(z_1,iz_1,\ldots,z_n,iz_n)\in L_{2n}\}$. And the latter is actually the Euclidean open unit ball of dimension $n$. 

The example above is inspired by the description of all complete totally geodesic complete manifolds in the Lie balls with respect to the Bergman metric in \cite[Theorem 2.3]{Fab 1990}.  Such manifolds are (up to an automorphism of $L_n$) the $L_k\times\{0\}^{n-k}$'s for some $k=1,\ldots,n$ and the sets $M$ as defined above. It would be interesting to see whether we could extend our results to all the Lie balls.

Before we get into the special situation of the tetrablock we make a tour to a general discussion of lifting of holomorphic retracts that is then used to get the form of some of holomorphic retracts.

\section{Lifting of retracts} Let $D$ and $G$ be two domains in $\mb C^n$ and $\pi:D\to G$ be a proper holomorphic mapping of multiplicity $k$. We denote {\it the set of critical arguments} by $\mathcal J(\pi):=\{z:\operatorname{det}\pi^{\prime}(z)=0\},$ where $\pi^{\prime}$ is the complex jacobian matrix of $\pi.$ The set $\pi(\mathcal J(\pi))$ is called {\it the locus set}. 

\begin{proposition}\label{proposition:lifting-retracts} Let $\pi:D\to G$ be as above and let $M\subset G$ be a simply connected retract such that $M\cap \pi(\mathcal J(\pi))=\emptyset$ and $R:G\to M$ be the corresponding retraction. Suppose that $\pi^{-1}(w)=\{z_1,\ldots,z_k\}$ for a fixed $w\in M$. For each $z_j,$ there exists a retract $z_j\in N_j\subset D$ such that $\pi(N_j)=M$ and $\pi_{|N_j}:N_j\to M$ is a biholomorphism.


\end{proposition}
\begin{proof}
Using the lifting theorem for the (unbranched) covering map $\pi : \pi^{-1}(M)\to M$ and $id_M:M\to M,$ we obtain an injective holomorphic mapping $i_j: M \to \pi^{-1}(M)$ such that $\pi \circ i_j=id_M$, $i_j(w)=z_j$ \cite[p. 143, 4.1. Theorem]{Bredon 1993}. 
Moreover, $i_j(M)$ is a holomorphic retract of $D.$ Actually, the retraction $R_j:D\to i_j(M)$ is given by 
\begin{eqnarray*}
    R_j(z)= i_j \circ R\circ\pi(z),\,\,z\in D.
\end{eqnarray*}  
\end{proof}

\begin{remark}
    If $G$ is simply connected then any retract $M\subseteq G,$ described in Proposition \ref{proposition:lifting-retracts}, is also simply connected.
\end{remark}
     A rather interesting application of Proposition \ref{proposition:lifting-retracts} is to provide a number of non-trivial holomorphic retracts of some special complex ellipsoids. 
    
    \begin{example} Let $p_j$ be positive integers for $j=1,\ldots,n$ and at least one of $p_j\geq 2.$ We define $$\mathcal E(p_1,\ldots,p_n):=\{(z_1,\ldots,z_n) \in \mb C^n: |z_1|^{2p_1}+|z_2|^{2p_2}+\ldots+|z_n|^{2p_n}<1\}$$ and a proper holomorphic map $\pi:\mathcal E(p_1,\ldots,p_n) \to \mb B_n$ by $\pi(z_1,\ldots,z_n)=(z_1^{p_1},\ldots,z_n^{p_n}).$ Note that $\mathcal J(\pi)=\{(z_1,\ldots,z_n) \in \mathcal E(p_1,\ldots,p_n): z_i=0 \text{ for at least one } i \text{ such that } p_i\geq 2\}.$ Any holomorphic retract in $\mathbb B_n$ is the intersection of an affine subspace with $\mathbb B_n$. Fix $k<n.$ We consider a $k$-dimensional affine space $V$ of $\mb C^n$ such that $V\cap\mathbb B_n$ does not intersect $\pi(\mathcal J(\pi))$. Let us define the set $$N:=\{(z_1^{1/p_1},\ldots,z_n^{1/p_n})=:\Phi(z) \,\, \text{ for } z\in V\cap\mathbb B_n\},$$ where the powers are well defined and holomorphic. Then $N$ is a holomorphic retract in $\mathcal E(p_1,\ldots,p_n),$ cf. Proposition \ref{proposition:lifting-retracts}. The retraction $R:\mathcal E(p_1,\ldots,p_n)\to N$ is given by $$R=\Phi\circ r \circ \pi,$$ where $r:\mb B_n \to V\cap\mathbb B_n$ denotes a holomorphic retraction.  Moreover, Proposition \ref{proposition:lifting-retracts} ensures that such retracts in $\mathcal E(p_1,\ldots,p_n)$ are biholomorphic to some $k$-dimensional Euclidean balls. 
\end{example}

\section{Retracts in the tetrablock}
This section is devoted to the description of (two dimensional) holomorphic retracts in the tetrablock.

\subsection{The Tetrablock}
Let $R_{III}(2)$ be the set of all $2\times2$ symmetric complex matrices whose largest singular value is less than $1.$  $R_{III}(2)$ is a bounded domain in $\mb C^3$ in its Harish-Chandra realization. In fact, $R_{III}(2)$ and $L_3$ are holomorphically equivalent and the biholomorphism is given by $\psi: L_3 \to R_{III}(2)$ by  $$\psi(z_1,z_2,z_3)=\begin{bmatrix} z_1+iz_2 & z_3 \\ z_3 & -z_1+iz_2\end{bmatrix}.$$ Moreover, $\psi(L_2\times\{0\})=\mb D^2\times \{0\}$ \cite[Lemma 3(b)]{Coh-Col 1994}. We use Theorem \ref{l3retract2dim} and the above discussion to conclude that all two-dimensional retracts in $R_{III}(2)$ are holomorphically equivalent to $\mb D^2$ (use Theorem~\ref{l3retract2dim} and Proposition~\ref{proposition:lifting-retracts}).

The proper holomorphic map $\Lambda: R_{III}(2) \to \Lambda(R_{III}(2))$ defined by $$\Lambda(A)=(a_{11},a_{22},a_{11}a_{22}-a_{12}^2),\;\,\, A=\begin{bmatrix} a_{11} & a_{12} \\ a_{12} & a_{22}\end{bmatrix}\in R_{III}(2)$$ is of multiplicity $2.$ Note that $\mathcal J(\Lambda)$ is the set of diagonal matrices with both entries from the unit disc. Consequently, $\Lambda(\mathcal J(\Lambda))=\{(a,b,ab):a,b\in\mathbb D\}=:\mathcal R$. The set $\mathcal R$ is sometimes called the {\it royal variety}. Geometrically, it is a bidisc. We refer to the proper image $\Lambda(R_{III}(2)):=\mb E$ by the \emph{tetrablock}. The domain was first introduced in \cite{Abo-You-Whi 2007}. Recall that $\mathbb E$ is a Lempert domain (see \cite{Edi-Kos-Zwo 2013}).

\subsection{Two-dimensional holomorphic retracts of the tetrablock}\label{proof:omitting royal variety}

The following theorem describes the two dimensional holomorphic retracts of the tetrablock (up to biholomorphisms). 

\begin{theorem}\label{theorem:2dimensional-retracts-tetrablock}
    Let $M\subset \mathbb E$ be a two-dimensional holomorphic retract of $\mathbb E$. Then $M$ is biholomorphic to the bidisc or to the symmetrized bidisc. \end{theorem}

We divide the proof of Theorem \ref{theorem:2dimensional-retracts-tetrablock} in two parts. We first remark that any holomorphic retract $M$ omitting the royal variety must be of the form $\Lambda(N)$, where $N$ is a two-dimensional holomorphic retract in $R_{III}(2)$ such that $N\cap\mathcal J(\Lambda)=\emptyset.$ Otherwise, it is of the form $\Phi(N),$ where $\Phi$ is an automorphism of $\mathbb E$ and $N$ is a two-dimensional holomorphic retract of one of the forms: $N=\{(a,b,ab):a,b\in\mathbb D\}$ or $N=\left\{\left(\frac{s}{2},\frac{s}{2},p\right):(s,p)\in\mathbb G_2\right\}$.


\begin{proof}[Proof for the retracts omitting the royal variety]
Since the tetrablock is starlike with respect to the origin \cite{Abo-You-Whi 2007}, any retract of the tetrablock is simply connected. Then by Proposition~\ref{proposition:lifting-retracts} any holomorphic retract in $\mathbb E$ omitting the royal variety must be a biholomorphic image (by $\Lambda$) of a retract from $R_{III}(2)$. This finishes the proof of this part of theorem.
\end{proof}

Note that the existence of two-dimensional holomorphic retracts in the tetrablock omitting the royal variety is yet to be proved. However, we provide examples of the two-dimensional retracts in the domain $R_{III}(2)$ omitting $\mathcal J(\Lambda).$ We emphasize that the existence of such retracts in $R_{III}(2)$ is necessary if the retracts in the tetrablock (omitting the royal variety) exist. We construct them below.


\begin{example}
  Recall that for $\lambda\in\mathbb D,$
    \begin{eqnarray*}
        \Psi_{\lambda}(A)=\left(A-\begin{bmatrix} 0 & \lambda  \\  \lambda & 0
\end{bmatrix}\right)\left(\mathbb I_2-\begin{bmatrix} 0 & \overline{\lambda} \\\overline{\lambda} &  0  
        \end{bmatrix} A\right)^{-1}, \,\,\, A\in\mathcal R_{III}(2)
    \end{eqnarray*} is an automorphism of $\mathcal R_{III}(2)$  \cite{Bas 1983}.
    Let $M=\left\{\begin{bmatrix} a & 0 \\
     0 & b \end{bmatrix}: a,b \in \mb D\right\}.$ Elementary calculations show that if $\lambda\neq 0,$  the set $\Psi_{\lambda}(M)$ does not intersect $\mathcal J(\Lambda)$.  
\end{example}


The above lifting of two-dimensional retracts is analogous to the lifting of complex geodesics omitting the royal varieties in the symmetrized bidisc and the tetrablock \cite{Pfl-Zwo 2005,Edi-Kos-Zwo 2013}. In particular, it is proved in \cite{Pfl-Zwo 2005} that the complex geodesics (one dimensional retracts) omitting the royal variety in the symmetrized bidisc are precisely the images under symmetrization map of complex geodesics in $\mathbb D^2$ omitting the diagonal of $\mathbb D^2.$ A similar understanding for the retracts (both one and two dimensional) of the tetrablock will be an interesting direction to pursue. 



 The problem of uniqueness of holomorphic retractions in the polydisc was studied in \cite{Guo-Hua-Wang 2008}. Another direction of exploration would be the uniqueness problem of holomorphic retractions for a holomorphic retract in $\mb E$. This can be thought of a natural generalization of the problem of uniqueness of left inverses to complex geodesics  \cite{Kos-Zwo 2016}. The following would be a natural sufficient condition for the uniqueness of holomorphic retracts. Let $M\subset D$ be a holomorphic retract such that the linear retraction $I_D(z)\cap T_zM\subset I_D(z)$ for some (any) $z\in M$ is unique. Does it follow that $M$ admits the unique retraction?



\subsection{Linear retracts in $I_{\mathbb E}(0)$}
We proved in Subsection \ref{proof:omitting royal variety} that the retracts in $\mathbb E$ omitting the royal variety are bidiscs (up to biholomorphisms). Here our aim is to determine all two-dimensional retracts of $\mb E$ intersecting the royal variety. In order to do that, we reduce our consideration to retracts passing through the origin by composing with a suitable automorphism of $\mb E$ \cite{You 2008}. Then we proceed as follows. First we recall the explicit description of the indicatrix of the tetrablock at $0$ from \cite{Abo-You-Whi 2007}. For $z=(z_1,z_2,z_3)\in\mathbb C^3,$ we have $\kappa_{\mathbb E}(0;z)=\max\{|z_1|+|z_3|,|z_2|+|z_3|\}$. The indicatrix of $\mb E$ at the origin is given by $$I_{\mathbb E}(0)=\{z\in\mb C^3:\kappa_{\mathbb E}(0;z)< 1\}.$$ A description of the linear retracts in the indicatrix $I_{\mathbb E}(0)$ is provided in Lemma \ref{linret}. We use it along with the description of complex geodesics passing through $0$ to understand the form of possible retracts, compare also Remark~\ref{remark:form-retracts}. Although we lack the uniqueness property of complex geodesics here, the complete description of complex geodesics in the tetrablock passing through $0$ in \cite{Edi-Zwo 2009} turns out to be a rather useful tool in our considerations.

We now provide a complete description of linear retracts in $I_{\mathbb E}(0)$. We start with the following technical lemma.

\begin{lemma}\label{Lemma estimate} Fix $\alpha,\beta\in\mathbb C$. Then
\begin{align}
|1+\alpha(\lambda-1)|\leq 1 \text{ for all $|\lambda|\leq 1,$}\text{ if and only if $\alpha\in[0,1]$ and} \label{equation-1}\\
|1+\alpha(\lambda-1)|+|\beta(\lambda-1)|\leq 1 \text{ for all $|\lambda|\leq 1,$}\text{ if and only if $\alpha\in[0,1]$, $\beta=0$.}\label{equation-2}
\end{align}
\end{lemma}
\begin{proof}
    The inequality $|1+\alpha(\lambda-1)|\leq 1$ is equivalent to
    \begin{eqnarray*}
2\re(\alpha(\lambda-1))+|\alpha|^2|\lambda-1|^2&\leq & 0,\\
2\re\left(\alpha\frac{1}{\overline{\lambda}-1}\right)+|\alpha|^2&\leq & 0.
    \end{eqnarray*}
   The above property implies that $\alpha\geq 0$. We put $\lambda=-1$ to conclude $\alpha\leq 1$. Conversely, it is straightforward to verify that the inequality holds for $\alpha\in[0,1]$. This finishes the proof of Equation \eqref{equation-1}.
    
    To verify Equation \eqref{equation-2}, we first observe that $\alpha\in[0,1]$. It is easy to see that $|\beta|\leq 1/2$. Under these assumptions, the inequality is equivalent to
    \begin{eqnarray*}
        (\alpha^2-|\beta|^2)|\lambda-1|+2\alpha\re\frac{\lambda-1}{|\lambda-1|}+2|\beta|\leq 0, \;\,\,\, |\lambda|\leq 1, \lambda\neq 1.
    \end{eqnarray*}
    Taking $\lambda\to 1,$ we get $|\beta|\leq 0.$
\end{proof}

For notational simplicity, we denote $K:=\overline{I_{\mathbb E}(0)}=\{z\in\mb C^3:\kappa_{\mathbb E}(0;z)\leq 1\}$.
We use the above description to deliver an ad hoc proof of the description of all two dimensional linear retracts in $K$. Though formally we go beyond the class of domains when considering linear retracts the understanding of linear retracts in the compact set $K$ should be evident.

\begin{lemma}\label{linret}
    Let $V$ be a two dimensional linear subspace of $\mathbb C^3$ and $R:K\to V\cap K$ be a linear retraction. Then $V$ is either $\mathbb C^2\times\{0\}$ or  ${\rm span}_{\mb C}\{(0,0,1),(1,\alpha,0)\}$ or ${\rm span}_{\mb C}\{(0,0,1),(\alpha,1,0)\}$ for some $|\alpha|\leq 1.$
\end{lemma}
\begin{proof}
 Assume that the two dimensional linear subspace $V$ is not $\mathbb C^2\times\{0\}$. Then $V\cap (\mathbb C^2\times\{0\})$ is one-dimensional. Let it be spanned by the vector $(1,\alpha,0)$ (or $(\alpha,1,0)$). Without loss of generality, we assume that it is spanned by $(1,t,0)$ for some $t\in[0,1]$. Since $R:K \to V\cap K$ is a linear retraction, we get $R(1,t,0)=(1,t,0)$.

Let $P_{(1,2)},P_3:K\to K$ denote the projection maps defined by $P_{(1,2)}(x_1,x_2,x_3)=(x_1,x_2,0)$ and $P_3(x_1,x_2,x_3)=(0,0,x_3).$ Let $R_{(1,2)}=P_{(1,2)}R$ and $R_3=P_3R.$ We claim that the mapping $R_{(1,2)}: \mb \overline{\mathbb D}^2\times\{0\} \to K$ defined by $(\lambda_1,\lambda_2,0)\mapsto P_{(1,2)}R(\lambda_1,\lambda_2,0)$ is a linear one dimensional retraction. We prove it by showing that $R_3(\lambda_1,\lambda_2,0)=(0,0,0)$ for all $\lambda_1,\lambda_2\in \mb D$. However,
\begin{eqnarray*}
     R(\lambda_1,\lambda_2,0) &=& R(1,t,0) + R(\lambda_1-1,\lambda_2-t,0)  \\&=&(1,t,0)+ (\lambda_1-1)(\alpha_1,\alpha_2,\alpha_3) + (\lambda_2-t)(\beta_1,\beta_2,\beta_3),
\end{eqnarray*}
where $(\alpha_1,\alpha_2,\alpha_3)=R(1,0,0)$ and $(\beta_1,\beta_2,\beta_3)=R(0,1,0)$.
Since $R(\lambda_1,\lambda_2,0)\in K,$ from Lemma \ref{Lemma estimate} we have  that $\alpha_1\in[0,1]$ and $\alpha_3=0$  (taking $\lambda_2=t$) and $\beta_1\in [0,1]$ (taking $\lambda_1=1$). 

Now if $t=1$ then (taking $\lambda_1=1$) $\beta_3=0$ from Lemma \ref{Lemma estimate} and if $0\leq t<1$ then we get in an elementary way that $|1+\beta_1(\lambda-1)|+|\beta_3||\lambda-1|\leq 1$ for all $|\lambda|\leq 1.$ This concludes that $\beta_3=0$. This finishes the proof of our claim. Also, $R_{(1,2)}=R|_{\mb D^2\times\{0\}}.$ So $R_{(1,2)}: \mb D^2\times\{0\} \to \mb D^2\times\{0\}$ is a one-dimensional linear retraction. Additionally, we know that $R(\mathbb C^2\times\{0\})=\mathbb C^2\times\{0\}$.

To finish the proof we show that $R(0,0,1)=(0,0,1)$. To see it, we consider a non-zero vector $v:=(0,\gamma_2,\gamma_3)\in V\cap (\{0\}\times\mathbb C^2).$ Clearly, $\gamma_3\neq 0$. We get
\begin{equation*}
    (0,\gamma_2,\gamma_3)=R(0,\gamma_2,\gamma_3)=\gamma_2R(0,1,0)+\gamma_3R(0,0,1).
\end{equation*}
Since $R(0,1,0)\in\mathbb C^2\times\{0\}$ and $R(0,0,1)\in K$ we get that $R(0,0,1)=(0,0,1).$ \end{proof}

\begin{remark} As a consequence of Lemma \ref{linret}, we obtain that the linear retracts of $I_{\mathbb E}(0)$ are of the following form (up to a linear automorphism of $\mathbb E$).
\begin{enumerate}
    \item[1.] $\mathbb D^2\times\{0\}$ and
    \item[2.] $V_t=\{(z_1,t z_1,z_3):|z_1|+|z_3|<1\}$ for some $t\in[0,1].$
\end{enumerate}
Moreover, the associated linear retractions of $I_{\mathbb E}(0)$ can be given by 

\begin{itemize}
    \item[1.] $R:I_{\mathbb E}(0) \to \mathbb D^2\times\{0\}$ such that $R(z_1,z_2,z_3)=(z_1,z_2,0)$ or 
    \item[2.] for some $t\in[0,1]$, $R_t:I_{\mathbb E}(0) \to V_t$ such that $R_t(z_1,z_2,z_3)= (z_1,t z_1,z_3).$
    \end{itemize}
\end{remark}

\subsection{Proof of Theorem~\ref{theorem:2dimensional-retracts-tetrablock}}

Here we present the proof of Theorem~\ref{theorem:2dimensional-retracts-tetrablock} for the retracts $M$ of $\mb E$ which intersect the royal variety $\m R$. Without loss of generality, we assume $0\in M$. The above discussion yields that $T_0M\cap I_{\mb E}(0)$ is either $\mathbb D^2\times\{0\}$ or $V_t$, $t\in [0,1]$.
We employ the method of recovering the necessary form of the holomorphic retract from that of the linear retract of $I_{\mb E}(0)$ as described in Remark~\ref{remark:form-retracts}. 

First we consider $T_0M\cap I_{\mb E}(0)=V_t$ or $V_t$ is a linear retract of $I_{\mathbb E}(0)$, where $t\in[0,1)$, and we show that it is not possible. Let $R$ be the corresponding retraction of $M$. Then $d_0R$ is the linear retraction of $I_{\mathbb E}(0)$ to $V_t$ given by the formula $d_0R(\lambda,\lambda,\mu)=(\lambda,t\lambda,\mu)$, $|\lambda|+|\mu|<1$. 

Recall that the left inverses in $\mathbb E$ to complex geodesics passing through the origin in directions from $V_t$ are uniquely determined with the exception of directions $(0,0,1)$ and $(1,t,0)$ \cite[Theorem 6.3]{Kos-Zwo 2016}. Moreover, the (uniquely) determined left inverses are of the form $\Psi_{\omega}(z):=\frac{\omega z_3-z_2}{\omega z_1-1}$, $z\in\mathbb E$. Additionally, for any $|\omega|=1$ there is a direction as above for which the geodesic in the direction has the left inverse $\Psi_{\omega}$. By Proposition~\ref{proposition:necessary-retraction}, we get that $\Psi_{\omega}\circ R(z)=\Psi_{\omega}(z)$, $z\in\mathbb E$, $|\omega|=1$, which easily implies that $R_j(z)=z_j$, $j=1,2,3$, $z\in\mathbb E$ and this gives a contradiction.

If the linear retract of $I_{\mathbb E}(0)$ corresponding to a retraction $R$ of $M$  is $\mathbb D^2\times\{0\},$ we use the description of all complex geodesics in $\mathbb E$ passing through $0$ from \cite{Edi-Zwo 2009} to conclude that $M\subset\mathcal R.$ Then the equality must hold and the map $z\to (z_1,z_2,z_1z_2)$ is a retraction of $\mathcal R$.

If the linear retract of $I_{\mathbb E}(0)$ is $V_1,$ we get by the description of complex geodesics passing through $0$ that $M$ must equal is $\{(s/2,s/2,p): (s,p)\in\mathbb G_2\}$.
Then a retraction of $\mathbb E$ is given by \begin{eqnarray*}\label{retalpha}R(z_1,z_2,z_3) = \left(\frac{z_1+z_2}{2},\frac{z_1+ z_2}{2},z_3\right).\end{eqnarray*} 

\subsection{Concluding remarks}
The problem of searching for all holomorphic retracts is somewhat difficult. However, in general situation a less general problem may be considered. Which domains $G\subset \mathbb C^m$ may be retracts of the given domain $D$ in the following sense. There exist holomorphic mappings $r:D\to G$, $R:G\to D$ such that $R\circ r$ is the identity. In the case of $D$ being a Lempert domain the only one-dimensional retract is the unit disc. In the case of the unit ball $\mathbb B_n$ only lower dimensional balls $\mathbb B_m,\,\,m\leq n,$ are its retracts. Similarly the only retracts of $\mathbb D^n$ are the lower dimensional polydiscs $\mathbb D^m,\,\,m\leq n$.

Our considerations led us to formulating that the only two-dimensional retracts in the Lie ball $L_3$ is the bidisc whereas in the case of the tetrablock there are two retratcs: the bidisc and the symmetrized bidisc. It seems to be a highly non-trivial task to get a classification in a much more general situation.



\end{document}